\documentclass[12pt,reqno]{amsart}
\usepackage[margin=1in]{geometry}
\makeatletter
\def\section{\@startsection{section}{1}%
	\z@{.7\linespacing\@plus\linespacing}{.5\linespacing}%
	{\bfseries
		\centering
}}
\def\@secnumfont{\bfseries}
\makeatother
\usepackage{amsmath,amssymb,amsthm,graphicx,amsxtra, setspace}
\usepackage[utf8]{inputenc}
\usepackage{charter}
\usepackage{mathrsfs}
\usepackage{alltt}
\usepackage{stackengine}
\usepackage{relsize}
\usepackage{hyperref}
\usepackage{aliascnt}
\usepackage{tikz}
\usepackage{mathtools}
\usepackage{multicol}
\usepackage{upgreek}
\usepackage{graphicx,type1cm,eso-pic,color}
\allowdisplaybreaks
\usepackage{scalerel,stackengine}
\stackMath

\newtheorem{theorem}{Theorem}[section]
\newtheorem*{theorem*}{Theorem}

\newcommand\dhookrightarrow{\mathrel{%
  \ensurestackMath{\stackanchor[.1ex]{\hookrightarrow}{\hookrightarrow}}
}}
\newaliascnt{lemma}{theorem}
\newtheorem{lemma}[lemma]{Lemma}
\aliascntresetthe{lemma} 

\newaliascnt{proposition}{theorem}

\aliascntresetthe{proposition}

\newaliascnt{assumption}{theorem}
\newtheorem{assumption}[assumption]{Assumption}
\aliascntresetthe{assumption}

\newaliascnt{auxiliary}{theorem}
\newtheorem{auxiliary}[auxiliary]{Auxiliary Result}
\aliascntresetthe{auxiliary}

\newaliascnt{corollary}{theorem}
\newtheorem{corollary}[corollary]{Corollary}
\aliascntresetthe{corollary}

\newaliascnt{definition}{theorem}
\newtheorem{definition}[definition]{Definition}
\aliascntresetthe{definition}

\newaliascnt{example}{theorem}

\aliascntresetthe{example}

\newaliascnt{remark}{theorem}
\newtheorem{remark}[remark]{Remark}
\aliascntresetthe{remark}

\newaliascnt{hypothesis}{theorem}

\aliascntresetthe{hypothesis}

\newaliascnt{property}{theorem}

\aliascntresetthe{property}

\DeclareMathOperator\supp{supp}
\usepackage[utf8]{inputenc}
\usepackage{amsmath}
\usepackage{enumerate}
\usepackage{bm}
\usepackage[T1]{fontenc}
\usepackage{amsthm}
\usepackage{amsfonts}
\usepackage{amssymb}
\usepackage{graphicx}
\usepackage{enumerate}
\usepackage[english]{babel}
\usepackage{comment}
\usepackage{amssymb}
\usepackage{amsthm}
\usepackage{xcolor}
\usepackage{dutchcal}

\newcommand{\Addresses}{{
		\footnote{
				\footnotesize
\noindent \textsuperscript{1}School of Mathematics, Indian Institute of Science Education and Research, Trivandrum (IISER-TVM),
			Maruthamala PO, Vithura, Thiruvananthapuram, Kerala, 695 551, INDIA  \par\nopagebreak \noindent
			\textit{e-mail:} \texttt{ritabrata20@iisertvm.ac.in}
			
			\noindent \textsuperscript{*}Corresponding author.

			\medskip\noindent
			{\bf Acknowledgments:} 
 			The author was supported by Prime Minister's Research Fellowship when this work was carried out. The author also like to thank Dr. Dhanya R. for valuable input and discussions on this problem.
			
}}}
\begin{document}
\title[Symmetry and Monotonicity]{ Symmetry and Monotonicity Property of a Solution of (p,q) Laplace Equation with Singular Term 	\Addresses	}
	\author[ Ritabrata Jana ]
	{ Ritabrata Jana\textsuperscript{1}} 
\maketitle
\begin{abstract}
This paper examines the behavior of a positive solution $u\in C^{1,\alpha}(\Bar{\Omega})$ of the $(p,q)$ Laplace equation with a singular term and zero Dirichlet boundary condition. Specifically, we consider the equation:
\begin{equation*}
\begin{aligned}
-div(|\nabla u|^{p-2}\nabla u+ a(x) |\nabla u|^{q-2}\nabla u) &= \frac{g(x)}{u^\delta}+h(x)f(u) \, &\text{in} \thinspace B_R(x_0), \quad
u & =0 \ &\text{on} \ \partial B_R(x_0).
\end{aligned}
\end{equation*}
We assume that $0<\delta<1$, $1<p\leq q<\infty$, and $f$ is a $C^1(\mathbb{R})$ nondecreasing function. Our analysis uses the moving plane method to investigate the symmetry and monotonicity properties of $u$. Additionally, we establish a strong comparison principle for solutions of the $(p,q)$ Laplace equation with radial symmetry under the assumptions that $1<p\leq q\leq 2$ and $f\equiv1$.
\end{abstract}
\section{Introduction}
In  this work, we analyse the behaviour of the solutions of the Dirichlet problem for quasilinear equation involving a singular nonlinearity of the form 
\begin{equation} \label{1.1}
\left\{
\begin{aligned}
-div(|\nabla u|^{p-2}\nabla u+ a(x) |\nabla u|^{q-2}\nabla u)  &= \frac{g(x)}{u^\delta}+h(x)f(u) \, &\text{in} \, \Omega\\
u & >0 \, &\text{in}  \, \Omega \\
u & =0 \ &\text{on} \ \partial \Omega
\end{aligned}   
\right.
\end{equation}
where $0<\delta<1$ and $1< p \leq q <\infty$.  Throughout this paper, we assume the domain to be $\Omega \equiv B_R(x_0)$ for some $x_0\in \mathbb{R}^N$, $R>0$ and $B_R(x_0)$ is defined as 
\begin{equation*}
    \begin{aligned}
        B_R(x_0) = \{x\in \mathbb{R}^N : |x-x_0|<R \}.
    \end{aligned}
\end{equation*}
We investigate a generalized differential operator that is denoted as the $(p,q)$ Laplacian, given by $-div(|\nabla u|^{p-2}\nabla u+ a(x) |\nabla u|^{q-2}\nabla u)$, where $a(x)\in C(\Bar{\Omega})$.  This operator has important applications in several fields, such as biophysics \cite{Fif79}, plasma physics \cite{Wil87}, reaction-diffusion equations \cite{Ari94}, nonlinear elasticity \cite{Zhi86}, and modelling of elementary particles \cite{BFP98,Der64}. To gain a more comprehensive understanding of the applications of the operator mentioned, readers can refer to \cite{MM18} and the references therein.
\\
As described by Marcellini \cite{Mar91}, the energy functional associated with equations driven by the $(p,q)$ Laplacian has non-standard growth conditions of $(p,q)$ type and is of the form 
$I(u) =
\int_\Omega h(x, \nabla u(x)) dx$, where the energy density satisfies $|\xi|^p \leq |h(x,\xi)|\leq|\xi |^{q}+1$. These types of functionals find applications in the field of nonlinear elasticity, and are particularly relevant to homogenisation theory. One specific functional that belongs to this category is the double phase functional, defined as $u \mapsto (|\nabla u|^p + a(x)|\nabla u|^q )dx$. The modulating coefficient $a(x)$ controls the geometry of the composite made by two differential materials with hardening exponents $p$ and $q$. The functional is studied by Zhikov\cite{Zhi86}, Mingione et al.\cite{BCM18, CM16}, and Radulescu et al.  \cite{PRR18, PRR19} and many more.
The corpus of literature pertaining to the theoretical foundations of the existence, and regularity of nonhomogeneous quasilinear $(p,q)$ Laplace equations
of the form 
\begin{equation}\label{geneq}
    \begin{aligned}
   -div(|\nabla u|^{p-2}\nabla u+|\nabla u|^{q-2}\nabla u)  =f(x,u) \thickspace \text{in} \thinspace \Omega, \quad u=0 \thickspace \text{on} \thinspace \partial\Omega .     
    \end{aligned}
\end{equation}
is progressively expanding. The interior and boundary regularity theory for the equation \eqref{geneq} where $\Omega$ is a bounded domain and $f(x,u)\equiv f(x) \in L^{\infty}(\Omega)$, is well established after the celebrated work of Lieberman \cite{Lib88, Lib91} and the parallels that concerning $p$ Laplacian. When the domain is  $\mathbb{R}^N$, He-Li \cite{HL08} established the local Holder regularity result when $f(x)$ is locally uniformly bounded. Kumar et al. \cite{KRV20} proved the interior $C^{1,\alpha}$ regularity and uniform  boundedness of  each positive solution to the problem \eqref{geneq} when $f(x,u) \equiv \lambda u^{-\delta}+u^{r-1}$ where $0<\delta<1$ and $1<q<p<r \leq p^{*}$. Recently, Giacomoni et al.\cite{GKS21}
demonstrated that the solution to problem \eqref{1.1} exhibits $C^{1,\alpha}$ boundary regularity, provided that a particular condition is satisfied. This condition can be expressed succinctly as follows:
\begin{assumption}\label{h1}
For $\beta\geq 0$ and $\beta + \delta <1$; $f,g,h,a$ satisfy the following:
\begin{enumerate} 
    \item[($C_g$)]   For two non negative constants $C_1$,$C_2$, and  $\Omega_\rho:=\{x \in \Bar{\Omega}: d(x)< \rho \}$ for $\rho>0$; $g \in L_{loc}^\infty(\Omega)$ satisfies
    \begin{equation*}
        \begin{aligned}
        \frac{C_1}{d(x)^\beta} \leq g(x) \leq \frac{C_2}{d(x)^\beta} \hspace{1cm} \text{when}\ x\in \Omega_\rho
        \end{aligned}
    \end{equation*}
    \item[($C_{fh}$)] Let $h$ and $f$ satisfy  $h(x)f(t) \geq 0 \, \forall \, (x,t) \in \Bar{\Omega} \times \mathbb{R}^+$, with $f(0)=0$ then there exists a non zero constant $C_3$ and $C_4$ such that $h(x)\leq C_3,$ $f(t) \leq C_4(1+t)^r$ for which $p-1< r \leq p^*-1$ if $p < n$ otherwise $p-1<r < \infty$.  
    \item[($C_a$)] Let $a(x)$ satisfies $0 \leq a(x) \in W^{1,\infty}(\Omega) \cap C(\Bar{\Omega})$.
\end{enumerate}
\end{assumption}
For a more comprehensive understanding of the regularity theory for solutions to equations like \eqref{geneq}, readers are referred to the articles \cite{CM16, Lib91}.  In light of the latest advancements in the existence and regularity theories for the $(p,q)$ Laplace equation, we focus on the qualitative characteristics of the solutions. In this work, we study the symmetry and monotonicity properties of the solutions $u \in C^{1,\alpha}(\Bar{\Omega})$ to the problem \eqref{1.1}. The boundary regularity is guaranteed by \cite{GKS21} thanks to Assumption \ref{h1}. In this  article, we use the famous moving plane method introduced by Alexandrov\cite{Ale62} and Serrin\cite{Ser71} and subsequently improved in \cite{BN91,GNN79}. 
A modified version of the moving plane method is exploited to investigate the symmetry and monotonicity of the positive solutions of the Dirichlet problem with singular nonlinearity for the semilinear case in \cite{CES19,CMS17,Tro16}, and for quasilinear setting in \cite{ES20}. Some indispensable analytical tools for moving plane methods, such as Hopf's lemma \cite{PW21}, strong comparison principle (SCP) \cite{BE21}, are
recently established for $(p,q)$ Laplacian. 
\par
The novelty of our study is to surmount the challenge of proving radial symmetry of solutions in the presence of non-homogeneity of $(p,q)$ Laplacian, as well as non-constant functions $f, g, h$, and $a$.  In the following sections, we elaborate on the technical challenges associated with this issue and provide solutions. We must emphasize that even for the homogeneous $p$-Laplace operator, the thorough investigation of radial symmetry of solutions is primarily conducted when $g$ and $h$ are constants, due to their critical role in the underlying equations. Nonetheless, to ensure the smooth operation of the moving plane method, we impose some conventional conditions on $f, g, h,$ and $a$, which are presented in Assumption \ref{h2}.
\begin{assumption} \label{h2}
For some $c>0$, $f,g,h$ and $a$ satisfy the following:
\begin{enumerate}
\item [($H_f$)] $f: \mathbb{R} \rightarrow \mathbb{R}$ is a $C^1$ nondecreasing function. 
 \item[($H_g$)]  $g: \Omega \rightarrow \mathbb{R}$ is a $C^1(\Bar{\Omega})$ radially decreasing function such that $g(x) \geq c >0$.
  \item[($H_h$)] $h: \Omega \rightarrow \mathbb{R}$ is a $C^1(\Bar{\Omega})$ radially decreasing function such that $h(x) \geq c >0$.
  \item[($H_a$)] $a: \Omega \rightarrow \mathbb{R}$ is a $C^1(\Bar{\Omega})$ radially increasing nonnegative function.
\end{enumerate}
\end{assumption}
In this work, we initially present the pertinent variations of Hopf's lemma and comparison principles for the specific problem under consideration. We subsequently conduct a thorough investigation of the critical set of solutions, denoted by $Z_u$, which is defined as follows:
\begin{definition}
Suppose $u \in C^{1,\alpha}(\Bar{\Omega})$ is a weak solution of \eqref{1.1}. Then $Z_u:=\{x \in \Omega \ | \ Du(x)=0\}$ is known as the critical set of $u$.
\end{definition}
The critical set $Z_u$ plays a vital role in the moving plane method, especially when dealing with nonlinear operators like the $p$-Laplacian. Our primary objective is to establish that $Z_u$ has measure zero and that $\Omega\setminus Z_u$ is connected. To achieve this, we take advantage of the $C^1$ regularities of $f$, $g$, $h$, and $a$. Next, we employ the moving plane technique to attain the symmetry characteristic of solutions and, we state our main result as follows:
\begin{theorem} \label{main}Let $1<p\leq q<\infty$ and $u \in C^{1,\alpha}(\overline{B_R}(x_0))$  solves (\ref{1.1}). Assume that $f,g,h$ and $a$ satisfy Assumption \ref{h1} and Assumption \ref{h2}.  Then the solution $u$ is radial and radially decreasing in ${B_R}(x_0).$ 
\end{theorem}
Subsequently, we shift our focus towards a non-homogeneous $(p,q)$ Laplace equation with a singular nonlinearity, namely equation \eqref{thm16}, and demonstrate the existence of SCP under the condition $0 \leq h_1 \leq h_2$ but $h_1 \not \equiv h_2$. 
   \begin{theorem} \label{SCPfleqg}
Let $1<p \leq q \leq 2$ and  $h_1$, $h_2$ are continuous radial functions on $B_R\equiv B_R(0)$ such that $0 \leq h_1 \leq h_2$  but $h_1 \not \equiv h_2$ in $B_R$. Assume that $u_1,u_2 \in C^{1,\alpha}(\Bar{B_R})$ such that they are radially decreasing solutions of 
\begin{equation}\label{thm16}
    \begin{aligned}
            -\Delta_p u_1 -\Delta_q u_1= \frac{c}{u_{1}^\delta}+h_1(x) \; \text{in} \; B_R \hspace{0.49cm}\text{and}\hspace{0.49cm} u_1=0 \; \text{on} \; \partial B_R
            \\
                -\Delta_p u_2 -\Delta_q u_2= \frac{c}{u_{2}^\delta}+h_2(x)  \; \text{in} \; B_R \hspace{0.49cm}\text{and}\hspace{0.49cm} u_2=0 \; \text{on} \; \partial B_R.
    \end{aligned}
\end{equation}
Then $0< u_1 < u_2$ in $B_R$ and $ \frac{\partial u_2}{\partial \nu}<\frac{\partial u_1}{\partial \nu}<0$ on $\partial B_R.$
\end{theorem}
A similar result for the homogeneous case, specifically for the $p$-Laplacian, is previously discussed by Dhanya et al. \cite[Theorem 1.1]{DIJ23}. Although Dhanya et al. \cite{DIJ23} stipulate the radial characteristic of solutions as a condition within their SCP, evincing the radial nature of said solutions can prove to be a challenging endeavor. Fortunately, Theorem \ref{main}, presented in this paper, provides valuable insights into this matter, specifically for the $(p,q)$ Laplacian and its homogeneous version, the $p$-Laplacian.
\\
To standardize the notation, it is assumed throughout this paper that the functions $f$, $g$, $h$, and $a$ satisfy Assumption \ref{h1} and Assumption \ref{h2}. The organization of this paper is as follows:
In section \ref{Prlm}, we obtain the Hopf's Lemma, weak comparison principle for narrow domains and SCP suitable for our problem. Then in section \ref{zu}, we study the critical set of solution in detail.
The symmetry and monotonicity result is proved in section \ref{symandmon}. Theorem \ref{SCPfleqg} is proved in section \ref{SCPsec}. 
\section{Preliminaries}\label{Prlm}
The objective of this section is to establish the pertinent notations utilized in this paper and obtain some preliminary results that will serve as fundamental components as we progress.
	\\
	At the outset of this section, we derive a boundary lemma of the Hopf type, which resembles the one presented in \cite{ES20} and \cite{PS07}. We define $I_\delta(\partial \Omega)$ as the neighbourhood of a boundary with unique nearest point property; i.e. for each  $x \in I_\delta(\partial \Omega)$ , there exists a point $ p(x)\in \partial \Omega$ such that $d(x,p(x))= d(x,\partial \Omega)$. Here, we set a inward normal on the boundary $n(x)$ where $n(x):= \frac{x-p(x)}{|x-p(x)|}$. We state the theorem as follows.
	\begin{theorem}\label{hpfe}
Let $u \in C^{1,\alpha} (\Bar{\Omega})$ be a positive weak solution of (\ref{1.1}). Then for any $\beta > 0$, there exists a neighbourhood $I_\delta(\partial \Omega)$ of $\partial\Omega$ such that 
\begin{align*}
\frac{\partial u}{\partial \nu}(x) >0 \ \forall \ x \ \in I_\delta(\partial \Omega)
\end{align*}
whenever $\nu(x) \in \mathbb{R^N}$ with $\| \nu(x) \|$=1 and ($\nu(x),n(x)) \geq \beta$.
\end{theorem}
\begin{proof}
    Let $u \in C^1(\Bar{\Omega})$ be a weak solution of (\ref{1.1}) then $\frac{\partial u}{\partial n }> 0$ using \cite[Theorem 5.5.1]{PS07}. But $\nu(x)=a \Gamma(x) + b n(x)$ for any two positive real numbers $a$ and $b$, $\Gamma $ is tangential direction at $x$. Since $u=0$ on the boundary, $\frac{\partial u}{\partial \Gamma }\geq  0$. Hence we conclude the theorem. 
\end{proof}
The following reult is discussed in \cite[Appendix]{BE21}. 
\begin{auxiliary}
 Let us assume that $\mathcal{V}\subset \mathbb{R}^N$ be a connected open set and $u,v\in C^1(\mathcal{V})$ such that $|\nabla u| \neq 0$ and $|\nabla v| \neq 0$ on the whole of $\mathcal{V}$ for $1< p \leq q$ then one have the followings.
\begin{enumerate}
    \item[B.1] $|\mathcal{A}(x,u,\nabla u)-\mathcal{A}(z,v,\nabla v)| \leq c_1|\nabla u - \nabla v|$
    \item[B.2] $\langle \mathcal{A}(x,u,\nabla u)-\mathcal{A}(z,v,\nabla v), \nabla u - \nabla v \rangle \geq c_2 (|\nabla u|+|\nabla v|)^{p-2}|\nabla u-\nabla v|^2$
\end{enumerate} 
for any $x \in K$ where $K \subset \mathcal{V}$ is compact, constants $c_1, c_2 > 0$ and $\mathcal{A}(x,z,\xi):=|\xi|^{p-2}\xi+a(x)|\xi|^{q-2}\xi$. 
\end{auxiliary}
The employment of auxiliary results allows us to establish comparison principles for the (p,q) Laplacian. The literature concerning comparison principles is extensive. For comparison principles related to the p-Laplacian, we refer to \cite{Dam98} and \cite{DS04}. In \cite{BE21}, the authors proved some comparison principles for the $(p,q)$ Laplacian. In the following, we present our version of comparison principles for the $(p,q)$ Laplacian with singular nonlinearity, which will be advantageous in proving our main theorems.
\begin{definition}
If $u,v \in W^{1,\infty}(\Omega)$ and $A\subset \Omega$ then we define
\begin{align*}
    M_A=M_A(u,v)=\sup_A(|\nabla u|+|\nabla v|) \\
    m_A=m_A(u,v)=\inf_A(|\nabla u|+|\nabla v|)
\end{align*}
\end{definition}
We get our weak comparison principles using the test function method.
\begin{theorem} \label{wcp}
 $u,v \in W^{1,\infty}(\Omega)$ satisfy the following:
\begin{align*}
    -div(|\nabla u|^{p-2}\nabla u+ a(x) |\nabla u|^{q-2}\nabla u) - \Lambda f(u) \leq -div(|\nabla v|^{p-2}\nabla v+ a(x) |\nabla v|^{q-2}\nabla v )- \Lambda f(v)
\end{align*}
where $\Lambda > 0$. Let $\Omega' \subset \Omega$ be open subset and suppose $u \leq v$ on $\partial \Omega'$. If $1<p<q<\infty$ and $f\in C^1(\Omega)$ is a increasing function then one of the following conditions is satisfied by $u$ and $v$
\begin{enumerate}
    \item[(C.1)] If $p=2$ there exists $\delta>0$ depending on $c_1,c_2$ in $(B.1)$ and $(B.2)$ such that if $|\Omega'|<\delta$ then $u \leq v$ in $\Omega'$.
    \item[(C.2)] If $1<p<2$ there exists $\delta,\mathcal{M}>0$ depending on $p,\Lambda, c_1, c_2,|\Omega| $ and $M_\Omega$ such that if $\Omega'=A_1 \cup A_2$ with $|A_1 \cap A_2|=0$, $|A_1|<\delta$ and $M_{A_2}<\mathcal{M}$ then $u\leq v$ in $\Omega'$.
    \item[(C.3)] If $p>2$ and $m_\Omega>0$ there exist  $\delta,m>0$ depending on $p,\Lambda, c_1, c_2,|\Omega| $ and $m_\Omega$ such that if $\Omega'=A_1 \cup A_2$ with $|A_1 \cap A_2|=0$, $|A_1|<\delta$ and $m_{A_2}>m$ then $u\leq v$ in $\Omega'$.
\end{enumerate}
\end{theorem}
\begin{proof}
    Since $f\in C^1(\mathbb{R})$ is increasing we get $0 \leq f(u)-f(v) \leq M (u-v)$  in the set $\{u \geq v\}$ for $M\geq 0$. We use $(u-v)^{+} \in W_0^{1,p}(\Omega')$ as a test function to get
    \begin{equation*}
    \begin{aligned}
    \int_{[u \geq v ]}(|\nabla u|^{p-2}\nabla u+ a(x) |\nabla u|^{q-2}\nabla u-|\nabla v|^{p-2}\nabla v- a(x) |\nabla v|^{q-2}\nabla v)(\nabla u -\nabla v )
    \\
 \leq \Lambda M \int_{[u \geq v]} (u-v)^2.
    \end{aligned}
\end{equation*}
For $c_2>0$ we have the following using the auxiliary result B.2 ,
\begin{equation} \label{wcpeqn}
    \begin{aligned}
\Lambda M    \int_{[u \geq v]} (u-v)^2 \geq  c_2 \int_{[u\geq v]}  (|\nabla u|+|\nabla v|)^{p-2}|\nabla u-\nabla v|^2.
    \end{aligned}
\end{equation}
First we want to prove $(C.1)$. Using \cite[Lemma 2.2]{Dam98} for $p=2$ and $\Omega'=\Omega'\cup \emptyset$ we have 
\begin{equation*}
    \begin{aligned}
c_2 \int_{\Omega'} |D(u-v)^{+}|^2 dx \leq \Lambda M    \int_{[u \geq v]} (u-v)^2  \leq \Lambda M \bigg(\frac{|\Omega'|}{\omega_{N}} \bigg) ^{\frac{2}{N}} \int_{\Omega'}|D(u-v)^{+}|^2
        \end{aligned}
\end{equation*}
where $\omega_N$ is the volume of a $N$ dimensional sphere. 
We chose $|\Omega'|$ small enough such that $c_2>\Lambda M \bigg(\frac{|\Omega'|}{\omega_{N}} \bigg) ^{\frac{2}{N}}$ to get $\|(u-v)^{+}\|_{W_0^{1,2}(\Omega')}=0$ so that  $(u-v)^{+}=0$ in $\Omega'$.
\\
Now we are going to prove $(C.2)$. Using \cite[Lemma 2.2]{Dam98} for $p=2$ and $\Omega'=A_1 \cup A_2$ with $|A_1 \cap A_2|=0$  we get
\begin{equation*}
    \begin{aligned}
        2 M \Lambda \omega_N^{\frac{-2}{N}} |\Omega'|^{\frac{1}{N}} \Bigg[ |A_1|^\frac{1}{N} \int_{A_1 \cap [u \geq v]} |\nabla u-\nabla v|^2 + |\Omega|^{\frac{1}{N}} \int_{A_2 \cap [u \geq v]} |\nabla(u-&v)|^2 \Bigg]
        \\
        &\geq \Lambda M    \int_{[u \geq v]} (u-v)^2
    \end{aligned}
\end{equation*}
Clearly for $\Omega'=A_1 \cup A_2$ with $|A_1 \cap A_2|=0$ we have
\begin{equation*}
    \begin{aligned}
        c_2 \int_{[u\geq v]}  (|\nabla u|+|\nabla v|)^{p-2}|\nabla u-\nabla v|^2 \geq c_2 M_\Omega^{p-2} &\int_{A_1 \cap [u \geq v]} |\nabla(u-v)|^2  
        \\
        &+ c_2 M_{A_2}^{p-2} \int_{A_2 \cap [u \geq v]} |\nabla(u-v)|^2 
    \end{aligned}
\end{equation*}
Hence we get the following from \eqref{wcpeqn}
\begin{equation*}
    \begin{aligned}
c_2 M_\Omega^{p-2} &\int_{A_1 \cap [u \geq v]} |\nabla(u-v)|^2  + c_2 M_{A_2}^{p-2} \int_{A_2 \cap [u \geq v]} |\nabla(u-v)|^2
        \\
        \leq & 2 M \Lambda \omega_N^{\frac{-2}{N}} |\Omega'|^{\frac{1}{N}} \Bigg[ |A_1|^\frac{1}{N} \int_{A_1 \cap [u \geq v]} |\nabla u-\nabla v|^2 + |\Omega|^{\frac{1}{N}} \int_{A_2 \cap [u \geq v]} |\nabla(u-v)|^2 \Bigg]
        \end{aligned}
\end{equation*}
So if $|A_1|$ and $M_{A_2}$ are small, then $\int_{A_i \cap [u \geq v]} |\nabla(u-v)|^2=0$ for $i=1,2$ so that $\|(u-v)^{+}\|_{W_0^{1,2}(\Omega')}=0$ and  $(u-v)^{+}=0$ in $\Omega'$.
\\
In the case of $(C.3)$, we get the same inequalities with $M_{S}$ replaced by $m_S$ where $S\equiv \Omega $ or $A_2$ for $p>2$.
\end{proof}
An SCP for sub and super solutions of the equation $-div (\mathcal{A}(x,u,\nabla u))+\Lambda u =0$ is established in \cite[Theorem A.1]{BE21}. However, we can replace this condition with \eqref{scpeqn1} and still obtain a similar result. The sub and super solution condition in \cite[Theorem A.1]{BE21} was only used to derive the inequalities $(8)$ and $(10)$ in \cite{Ser70}. Nevertheless, we can infer these inequalities using \eqref{scpeqn1}. Therefore, using a similar argument as in \cite[Theorem A.1]{BE21}, we can conclude a similar result and state the following.
\begin{theorem} \label{scp}
Let $\mathcal{V}\subset \mathbb{R}^N$ be a connected open set and let $u,v \in C^1(\mathcal{V})$ such that $u \leq v$ and satisfy the following
\begin{equation} \label{scpeqn1}
    \begin{aligned}
           -div(|\nabla u|^{p-2}\nabla u+ a(x) |\nabla u|^{q-2}\nabla u)  + \Lambda u \leq -div(|\nabla v|^{p-2}\nabla v+ a(x) |\nabla v|^{q-2}\nabla v)  +\Lambda v
    \end{aligned}
\end{equation}

We assume that $|\nabla u| \neq 0$ and $|\nabla v|\neq 0$ on the whole of $\mathcal{V}$.  Then either $u\equiv v$ or  $u < v$ throughout $\mathcal{V}$.
\end{theorem}
\section{Critical Set of Solution}{\label{zu}}
The critical set $Z_u$ of a solution plays a crucial role to get the monotonicity and symmetry properties of the solutions. In order to gain insight into some of the qualitative aspects of $Z_u$, we examine the $L_u$ operator, inspired by \cite{DS04}.
\begin{definition}
We define the operator $L_u$ as follows
\begin{equation*}
    \begin{aligned}
        L_u(\mathcal{g},\varphi) 
        \equiv \int_\Omega \Bigg[ |Du|^{p-2}(D\mathcal{g},D\varphi)&+(p-2)|Du|^{p-4}(Du,D\mathcal{g})(Du,D\varphi)
        +a|Du|^{q-2}(D\mathcal{g},D\varphi)\\&+a(q-2)|Du|^{q-4}(Du,D\mathcal{g})(Du,D\varphi)
        \\
        &+|Du|^{q-2}\frac{\partial a}{\partial x_i}(Du,D\varphi)-\varphi\frac{\partial}{\partial x_i}(\frac{g(x)}{u^\delta}+h(x)f(u))  \Bigg]dx 
    \end{aligned}
\end{equation*}
\end{definition}
The operator $L_u$ is well defined if $\mathcal{g} \in L^2(\Omega,\mathbb{R})$ and $|Du|^{s-2}D\mathcal{g} \in L^2(\Omega,\mathbb{R}^N)$ for $s=p$,$q$ and $\varphi \in W^{1,2}(\Omega)$. Unlike in \cite{DS04}, the operator $L_u$ is not linear due to the presence of non-constant functions $a$, $g$, and $h$. In this section, we carefully analyse the $L_u$ operator to prove that $Z_u$ is a set of measure zero and $\Omega\setminus Z_u$ is connected just like the homogeneous $p$-Laplace operator. 
\\
Since $u \in C^1(\Bar{\Omega})$ is positive in $\Omega$ and  $u$ satisfies a uniformly elliptic equation in a neighbourhood of each regular point $x\in \Omega\setminus Z_u$, we have $u\in C^2 (\Omega\setminus Z_u)$ by standard elliptic regularity similar to \cite[Remark 2.1]{DS04}. Now we fix the notations for the generalised derivative.
\begin{definition}\label{grad}
We set $\Tilde{u}_{ij}= u_{x_{i}x_{j}}$ in $\Omega\setminus Z_u$ and $0$ in $Z_u$. We use the notation $\Tilde{D}u_i$ for the "gradient" $(\tilde{u}_{i1},...,\tilde{u}_{iN})$. Similar to \cite{DS04}, if we can prove $|Du|^{s-2}u_{x_{i}}\in W^{1,2}(\Omega)$ then 
\begin{align*}
    \frac{\partial}{\partial x_j}(|Du|^{s-2}u_{x_i}) \equiv \bigg(|Du|^{s-2}\Tilde{u}_{ij}+(s-2)|Du|^{s-4}(Du,\Tilde{D}u_j)u_{x_i}\bigg)
\end{align*}
where $\frac{\partial}{\partial x_j}$ stands for distributional derivative for $s=p,q.$
\end{definition}
For the time being, we use the deﬁnition of the $L_u$ operator at the ﬁxed solution $u$ only with the test functions $\varphi \in W^{1,2}(\Omega)$ with compact support in $\Omega \setminus Z_u$ and get the following.
\begin{lemma} \label{lelu}
Let $u\in C^{1.\alpha}(\Bar{\Omega})$ be a weak solution of (\ref{1.1}) then we have $L_u(u_{x_i},\varphi)=0$ for every $\varphi \in W^{1,2}(\Omega)$ with compact support in $\Omega \setminus Z_u$.
\end{lemma}
\begin{proof}
Let $\varphi\in C_c^\infty(\Omega\setminus Z_u)$ then we have $\supp \frac{\partial \varphi}{\partial x_i} \subset \subset \Omega \setminus Z_u$. If we use $\frac{\partial \varphi}{\partial x_i}$ as a test function in \eqref{1.1} then the domain of the integration would be a compact subset of $\Omega\setminus Z_u$. Observe that $|Du|^{p-2}u_{x_i} $ lies in $ W_{loc}^{1,2}(\Omega \setminus Z_u)$ since $u\in C^2(\Omega\setminus Z_u ) $. Now by leveraging the smoothness of $f$, $g$, $h$, and $a$, as well as the positivity and regularity of the function $u$, we can apply integration by parts to equation \eqref{1.1} with test functions of the form $\frac{\partial \varphi}{\partial x_i}$. This yields $L_u(u_{x_i},\varphi)=0$ for $\varphi\in C_c^\infty(\Omega\setminus Z_u)$. Now similar to \cite[Lemma 2.1]{DS04} we get that $L_u(u_{x_i},\varphi)=0$ for $\varphi \in W^{1,2}(\Omega)$ with compact support in $\Omega\setminus Z_u$ using the density arguments. 
\end{proof}
\begin{lemma}
\label{lu}
Let $u\in C^{1,\alpha}(\Bar{\Omega})$ is a solution of (\ref{1.1}). Then for any $E\subset\subset\Omega$ and for every $i,j=1,2,..N$ , we have a constant $C$ depending on $\beta,\gamma,E$ such that, 
\begin{equation}\label{3.4}
   \begin{aligned}
    \sup_{x\in \Omega} \int_{E\setminus \{u_{x_i}=0\}} \frac{|Du|^{p-2}}{|u_{x_i}|^\beta|x-y|^\gamma} |\tilde{D}u_i|^2 dy+ a(x) \frac{|Du|^{q-2}}{|u_{x_i}|^\beta|x-y|^\gamma} |\tilde{D}u_i|^2 dy\leq C
   \end{aligned}
\end{equation}
where $\beta<1$ ,$\gamma<N-2$ if $N\geq 3$ and $\gamma=0 $ if $N=2$.
\end{lemma}
\begin{proof}
We can assume without loss of generality that $x\in E$. In order to prove \eqref{3.4}, it is sufficient to demonstrate the following for every measurable subset $E\subset\subset\Omega$,
\begin{align}\label{1}
    \sup_{x\in E} \int_{E\setminus \{u_{x_i}=0\}} \frac{|Du|^{p-2}}{|u_{x_i}|^\beta|x-y|^\gamma} |\tilde{D}u_i|^2 +a(x)\frac{|Du|^{q-2}}{|u_{x_i}|^\beta|x-y|^\gamma} |\tilde{D}u_i|^2 dy<C(\beta,\gamma,E). 
\end{align}
This is because we can construct the set $E_\delta = \{x\in \Omega: d(x,E) \leq \delta\}$ with $0<\delta< \frac{1}{2} d(x,\partial \Omega)$. Then, by considering both cases where $x\in E_\delta$ and $x\in \Omega \setminus E_\delta$, we can show that \eqref{1} implies that 
\begin{equation*} 
    \begin{aligned}
       \sup_{x\in \Omega} \int_{E\setminus \{u_{x_i}=0\}} \frac{|Du|^{p-2}}{|u_{x_i}|^\beta|x-y|^\gamma} |\tilde{D}u_i|^2 +a(x)\frac{|Du|^{q-2}}{|u_{x_i}|^\beta|x-y|^\gamma} |\tilde{D}u_i|^2 dy<C(\beta,\gamma,E_\delta)+\frac{1}{\delta^\gamma} C(\beta,0,E).
    \end{aligned}
\end{equation*}
Let $E \subset \subset \Omega $, $x\in E$ and we consider a cut-off function $\varphi \in C_c^{\infty}(\Omega)$ such that $\varphi \geq 0$ in $\Omega$, and $\varphi \equiv 1 $ in $E_\delta$. We define $G_\epsilon$ as 
\begin{equation}\label{defG}
    G_\epsilon(s)=\left\{\begin{aligned}
&0,  \hspace{1.2cm} \text{ in } \ |s| \leq \epsilon, \\
&2s-2\epsilon,  \ \text{ in } \ \epsilon \leq |s| \leq 2\epsilon , \\
&s  \hspace{1.5cm} \text{in} \  |s| \geq 2 \epsilon
\end{aligned}\right.
\end{equation}
so that $G_\epsilon$ is Lipschitz continuous function and $0 \leq G_\epsilon'\leq 2$.
We will separately consider the cases $x \in E \cap Z$ and $x \in E\setminus Z$ .
\\
For the case $x \in E \cap Z$, we  define $\psi_{\epsilon,x}(y):= \frac {G_{\epsilon}(u_{x_i})(y) \varphi(y)}{|u_{x_i}(y)|^\beta|x-y|^\gamma}$ with $\beta<1$ and $\gamma < N-2$ when $N \geq3$. For $N=2$ we use $\psi_{\epsilon,x}=\frac{G_{\epsilon}(u_{x_i}) \varphi}{|u_{x_i}|^\beta}$. 
Since $G_\epsilon(u_{x_i})$ vanishes in the neighbourhood of each critical point and also in a neighbourhood of $y=x$ , we can treat $\psi_{\epsilon,x}$ as test functions in $L_u(u_{x_i},\cdot)$ and get the following 
\begin{equation*}
    \begin{aligned}
  \sum_{s=p,q}\Bigg[   \int_\Omega a_s \frac{|Du|^{s-2}}{|u_{x_i}|^\beta} \frac{|\tilde{D}u_i|^2}{|x-y|^\gamma} &(G_\epsilon'(u_{x_i})-  \frac{\beta G_\epsilon(u_{x_i})}{u_{x_i}}) \varphi dy  \\ &+ (s-2) \int_\Omega a_s \frac{|Du|^{s-4}(Du,\tilde{D}u_i)^2}{|u_{x_i}|^\beta |x-y|^\gamma} (G_\epsilon'(u_{x_i})- \frac{\beta G_\epsilon(u_{x_i})}{u_{x_i}}) \varphi dy \\ &+\int_{\Omega \setminus E_\delta} a_s |Du|^{s-2} (\tilde{D}u_i,D\varphi) \frac{G_\epsilon (u_{x_i}) }{|u_{x_i}|^\beta |x-y|^\gamma} dy \\ &+(s-2) \int_{\Omega \setminus E_\delta} a_s \frac{|Du|^{s-4}(Du,\tilde{D}u_i)}{|u_{x_i}|^\beta |x-y|^\gamma}(Du,D\varphi) G_\epsilon (u_{x_i}) dy \\ & 
    + \int_\Omega a_s |Du|^{s-2}(\tilde{D}u_i, D_y \frac{1}{|x-y|^\gamma})\frac{G_\epsilon(u_{x_i})}{|u_{x_i|^\beta}} \varphi dy \\ &+ (s-2) \int_\Omega a_s |Du|^{s-4} (Du,\tilde{D}u_i)(Du, D_y \frac{1}{|x-y|^\gamma})\frac{G_\epsilon(u_{x_i}) \varphi}{|u_{x_i}|^\beta} dy\Bigg]\\
    & + \int_{\Omega} a_{x_{i}} |Du|^{q-2} (Du,D\psi_{\epsilon,x}) dy -\int_{\Omega} \big(\frac{g}{u^\delta}+f(u)h\big)_{x_i} \psi_{\epsilon,x} dy= 0
    \end{aligned}
\end{equation*}
where 
\begin{equation}
 a_s(x)=\left\{\begin{aligned}
1, & \ \text{ if } \ s=p  \\
a(x), & \ \text{ if } \ s=q
\end{aligned}\right.
\end{equation}
 By the definition of $G_\epsilon$, we have $(G_\epsilon'(u_{x_{i}})-\beta\frac{G_\epsilon(u_{x_{i}})}{u_{x_{i}}})\geq 0$ in $\Omega$.  
By utilizing the fact that $a(x)$ is non-negative and applying the triangle inequality, we obtain the following result, where $C_{1s}$, $C_{2s}$, $C_3$, and $C_4$ are positive constants for $s=p\text{ or } q $:      
\begin{equation} \label{estincrit}
    \begin{aligned}
   \sum_{s=p,q} \bigg(\int_\Omega & a_s  \frac{|Du|^{s-2}}{|u_{x_i}|^\beta} \frac{|\tilde{D}u_i|^2}{|x-y|^\gamma}  (G_\epsilon'(u_{x_i})-  \frac{\beta G_\epsilon(u_{x_i})}{u_{x_i}}) \varphi \bigg) 
    \\
\leq \sum_{s=p,q} \bigg( & C_{1s}\int_{\Omega \setminus E_\delta} |a_s| |Du|^{s-2} |\tilde{D}u_i||D\varphi|  \frac{G_\epsilon(u_{x_i}) }{|u_{x_i}|^\beta |x-y|^\gamma} dy 
\\
&+ C_{2s}\int_\Omega a_s |Du|^{s-2}|\tilde{D}u_i|  \frac{\gamma}{|x-y|^{\gamma+1}}\frac{G_\epsilon(u_{x_i})}{|u_{x_i|^\beta}} \varphi dy \bigg)
\\
&+C_3 \int_\Omega |Du|^{q-1}|\frac{\partial a}{\partial x_i}| |D \psi_{\epsilon,x}|
+C_4 \int_\Omega |(\frac{g}{u^\delta}+f(u))_{x_i}h| \frac{G_\epsilon(u_{x_i}) \varphi}{|u_{x_i}|^\beta |x-y|^\gamma}
    \end{aligned}
\end{equation}
Now $\supp \psi_{\epsilon,x}\subset\subset\Omega \setminus Z_u \subset \Omega$ and $u \in C^{1,\alpha}(\Bar{\Omega})$ as well as $u\in C^2(\Omega\setminus Z_u)$. We also have $u>0$ inside $\Omega$. Hence $\int_\Omega |Du|^{q-1}|\frac{\partial a}{\partial x_i}| |D \psi_{\epsilon,x}|$ and  $\int_\Omega |(\frac{g}{u^\delta}+f(u)h)_{x_{i}}| \frac{G_\epsilon(u_{x_i}) \varphi}{|u_{x_i}|^\beta |x-y|^\gamma}$ are bounded independent of $x$ using $C^1$ nature of $a,g,h$ and $f$. Given $E_\delta$ and $x \in E$, we know that
\begin{equation*}
    \begin{aligned}
        \sup_{y \in \Omega\setminus E_{\delta}} \frac{1}{|x-y|^\gamma}\leq\frac{1}{\delta^\gamma}.
    \end{aligned}
\end{equation*} 
Due to the compact support of the test functions in $\Omega$ and the fact that $|Du|^{s-2}|\tilde{D}u_{i}| \in L^2_{loc}(\Omega)$, we can conclude that the integral $\int_{\Omega \setminus E_\delta} |a_s(x)| |Du|^{s-2} |\tilde{D}u_i||D\varphi| \frac{G_\epsilon(u_{x_i}) }{|u_{x_i}|^\beta |x-y|^\gamma} dy$ is also bounded independently of $x$. In conclusion,
\begin{equation*}
    \begin{aligned}
        \int_\Omega |Du|^{q-1}|\frac{\partial a}{\partial x_i}| |D \psi_{\epsilon,x}|+\int_\Omega |(\frac{g}{u^\delta}&+f(u)h)_{x_{i}}| \frac{G_\epsilon(u_{x_i}) \varphi}{|u_{x_i}|^\beta |x-y|^\gamma}
        \\
        &+\int_{\Omega \setminus E_\delta} |a_s| |Du|^{s-2} |\tilde{D}u_i||D\varphi|  \frac{G_\epsilon(u_{x_i}) }{|u_{x_i}|^\beta |x-y|^\gamma} dy  \leq C
    \end{aligned}
\end{equation*} for some $C>0$.
\\
Using \eqref{estincrit}, for $C_5>0$ we have the following:
\begin{equation*}
     \begin{aligned}
      \sum_{s=p,q} \int_\Omega a_s & \frac{|Du|^{s-2}}{|u_{x_i}|^\beta} \frac{|\tilde{D}u_i|^2}{|x-y|^\gamma} (G_\epsilon'(u_{x_i})- \frac{\beta G_\epsilon(u_{x_i})}{u_{x_i}}) \varphi dy \\ \leq C_5 
+\sum_{s=p,q} & C_{2s} \int_\Omega  (a_s)^{\frac{1}{2}} \frac{ |Du|^{\frac{s-2}{2}}|\tilde{D}u_i|}{|x-y|^{\frac{\gamma}{2}}|u_{x_i}|^\frac{\beta}{2}}(\frac{|G_\epsilon(u_{x_i})|}{|u_{x_i}|} \varphi)^\frac{1}{2}     (a_s)^{\frac{1}{2}} \frac{|Du|^{\frac{s-2}{2}} |G_\epsilon(u_{x_i})|^\frac{1}{2} |u_{x_i}|^{\frac{1}{2} - \frac{\beta}{2}} \varphi^\frac{1}{2}}{|x-y|^{\frac{\gamma}{2}+1}}  dy 
     \end{aligned}
 \end{equation*}
  By Young's inequality, we get the following for $s=p,q$ and $\sigma_s >0$ which would be chosen later
 \begin{equation*}
     \begin{aligned}
      \int_\Omega  (a_s)^{\frac{1}{2}} \frac{ |Du|^{\frac{s-2}{2}}|\tilde{D}u_i|}{|x-y|^{\frac{\gamma}{2}}|u_{x_i}|^\frac{\beta}{2}}(\frac{|G_\epsilon(u_{x_i})|}{|u_{x_i}|} \varphi)^\frac{1}{2}      &(a_s)^{\frac{1}{2}} \frac{|Du|^{\frac{s-2}{2}} |G_\epsilon(u_{x_i})|^\frac{1}{2} |u_{x_i}|^{\frac{1}{2} - \frac{\beta}{2}} \varphi^\frac{1}{2}}{|x-y|^{\frac{\gamma}{2}+1}}  dy  \\ \leq &\sigma_s \int_\Omega a_s \frac{ |Du|^{s-2}|\tilde{D}u_i|^2}{|x-y|^{\gamma}|u_{x_i}|^{\beta}}(\frac{|G_\epsilon(u_{x_i})|}{|u_{x_i}|} \varphi) dy \\&+ \frac{1}{4 \sigma_s} \int_\Omega a_s \frac{|Du|^{s-2} |G_\epsilon(u_{x_i})| |u_{x_i}|^{1 -\beta} \varphi^\frac{1}{2}}{|x-y|^{\gamma+2}} dy
     \end{aligned}
 \end{equation*}
Observe that $|G_\epsilon(u_{x_i})| \leq 2| u_{x_i}|$, $\beta<1$, and $\frac{|G_\epsilon(u_{x_i})|}{| u_{x_i}|}\equiv\frac{G_\epsilon(u_{x_i})}{u_{x_i}}$. Now  we choose  $\sigma_s>0$ such a way that $(1-\beta-\sigma_s)>0$ and get the following  
 \begin{equation*}
     \begin{aligned}
    \sum_{s=p,q} \int_\Omega a_s\frac{|Du|^{s-2}}{|u_{x_i}|^\beta} \frac{|\tilde{D}u_i|^2}{|x-y|^\gamma} (G_\epsilon'(u_{x_i})-& \frac{(\beta+\sigma_s) G_\epsilon(u_{x_i})}{u_{x_i}}) \varphi dy 
     \\
     \leq C_5 + & C_6 \int_\Omega \frac{1}{|x-y|^{\gamma+2}} dy
     \leq K  
     \end{aligned}
 \end{equation*}
 where $K>0$ does not depend on $x$ and all constants are absorbed in $C_6>0$.  By definition $(G_\epsilon'(u_{x_i})- \frac{(\beta+\sigma_s) G_\epsilon(u_{x_i})}{u_{x_i}}) \geq 0 $ and $(G_\epsilon'(u_{x_i})- \frac{(\beta+\sigma_s) G_\epsilon(u_{x_i})}{u_{x_i}}) \longrightarrow 1-(\beta+\sigma_s)$ in $\{u_{x_i} \neq 0\}$. Therefore by Fatou's Lemma we get
 \begin{equation*}
     \begin{aligned}
    \sum_{s=p,q}  \int_{\Omega \setminus \{u_{x_i}=0\}} a_s \frac{|Du|^{s-2}}{|u_{x_i}|^\beta} \frac{|\tilde{D}u_i|^2}{|x-y|^\gamma} \varphi dy \leq K_1
     \end{aligned}
 \end{equation*}
 where $K_1>0$ does not depend on $x$. In particular,
 $\varphi \equiv 1$ in E and we get the result for the first case viz. $x\in E\cap Z_u$.
 \\
 For the second case, where $x\in E\setminus Z_u$ and $\epsilon>0$ small ,we define $\varphi_{\epsilon,x}\in C_c^\infty (\Omega)$ such that for some $\tilde{C}>0$ the followings hold. 
\begin{equation}
    \left. \begin{aligned}
     \varphi_{\epsilon,x} & \geq 0 \; \text{in} \;  \Omega \\
      \varphi_{\epsilon,x} & \equiv 0 \; \text{in} \;B_\epsilon(x)\\
      \varphi_{\epsilon,x} &\equiv 1 \; \text{in} \;  E_\delta \setminus B_{2\epsilon(x)} \\
      |D\varphi_{\epsilon,x}| & \leq \frac{\tilde{C}}{\epsilon} \; \text{in} \; B_{2\epsilon(x)} \setminus B_{\epsilon(x)} \\ 
      |D\varphi_{\epsilon,x}| & \leq \tilde{C} \;  \text{outside}\;  B_{2\epsilon(x)}
    \end{aligned}
    \right\}
\end{equation}
Suppose that there exists a set $A\subset \subset \Omega$ such that $\supp (\varphi_{\epsilon,x}) \subset A$ for each $\epsilon$ and $x\in E$ then we can use $\psi_{\epsilon,x}\equiv \frac{ G_{\epsilon}(u_{x_i}) \varphi_{\epsilon,x}}{|u_{x_i}|^\beta |x-y|^\gamma}$
as a test function in $L_u(u_{x_i},\cdot)=0$. Using the arguments of the previous case, for $C_7,C_8>0$
we have 
\begin{equation*}
    \begin{aligned}
            \sum_{s=p,q} \int_\Omega a_s\frac{|Du|^{s-2}}{|u_{x_i}|^\beta} \frac{|\tilde{D}u_i|^2}{|x-y|^\gamma} (G_\epsilon'(u_{x_i})-& \frac{(\beta+\sigma_s) G_\epsilon(u_{x_i})}{u_{x_i}}) \varphi_{\epsilon,x} dy 
            \\
            \leq C_7+ C_8 \sum_{s=p,q} &\int_{B_{2\epsilon(x)} \setminus B_{\epsilon(x)} } \frac{|Du|^{s-2}}{|u_{x_i}|^\beta} \frac{|\tilde{D}u_i|}{|x-y|^\gamma} G_\epsilon(u_{x_i})|D\varphi_{\epsilon,x}| \thickspace dy
    \end{aligned}
\end{equation*}
 Since $x \in E \setminus Z$ and $u \in C^2(E\setminus Z)$ then for $s=p,q$ and $\epsilon$ sufficiently small there exist  $C_{9s}(\epsilon,x)$ such that $|Du|^{s-2}|\tilde{D}u_i| \leq C_{9s}(\epsilon,x) $ in $B_{2 \epsilon}(x)$. Now if x is fixed and $\epsilon$ small then $C_{9s}$ does not depend on $\epsilon$ . We get the following for $\epsilon>0$ very small 
\begin{equation*}
    \begin{aligned}
          C_8 \sum_{s=p,q} &\int_{B_{2\epsilon(x)} \setminus B_{\epsilon(x)} } \frac{|Du|^{s-2}}{|u_{x_i}|^\beta} \frac{|\tilde{D}u_i|}{|x-y|^\gamma} G_\epsilon(u_{x_i})|D\varphi_{\epsilon,x}| \thickspace dy   \leq (C_{9p}(x)+C_{9q}(x)) \frac{\epsilon^N}{\epsilon^{\gamma+1}}
    \end{aligned}
\end{equation*}
Since $\gamma < N-2$  so $(C_{9p}(x)+C_{9q}(x)) \frac{\epsilon^N}{\epsilon^{\gamma+1}} \longrightarrow 0$ as $\epsilon \longrightarrow 0$. Now, we can arrive at our desired result by applying similar concepts as described in \cite[Page 497]{DS04}.
\end{proof}
As an easy consequence of the previous lemma, we get the following theorem.
\begin{theorem} 
Let $u\in C^{1,\alpha}(\Bar{\Omega})$ be a solution of \eqref{1.1}. Then for  $1<p<q<\infty$ and any $E\subset\subset\Omega$ and for every $i,j=1,2,..N$ , we have
\begin{equation*}
   \begin{aligned}
    \sup_{x\in \Omega} \int_{E\setminus \{u_{x_i}=0\}} \frac{|Du|^{s-2}}{|u_{x_i}|^\beta|x-y|^\gamma} |\tilde{D}u_i|^2 dy\leq C
   \end{aligned}
\end{equation*}
where $\beta<1$ ,$\gamma<N-2$ if $N\geq 3$ and $\gamma=0 $ if $N=2$ and $C\equiv C(\beta,\gamma,E)$ for $s=p$ or $q$. Moreover 
\begin{equation*}
    \sup_{x\in \Omega} \int_{E\setminus Z} \frac{|Du|^{s-2-\beta}}{|x-y|^\gamma} \|D^2u\|^2 dy\leq C
\end{equation*}
where $C>0$.
\end{theorem}
In the proof of Lemma \ref{lelu}, we get $ \int_\Omega a_s\frac{|Du|^{s-2}}{|u_{x_i}|^\beta} \frac{|\tilde{D}u_i|^2}{|x-y|^\gamma} (G_\epsilon'(u_{x_i})- \frac{(\beta+\sigma_s) G_\epsilon(u_{x_i})}{u_{x_i}}) \varphi dy \leq K $ where $K$ is a uniform bound not depending on $x$ or $\epsilon$ for $s=p$ or $q$. Also we have $(G_\epsilon'(u_{x_i})- \frac{(\beta+\sigma_s) G_\epsilon(u_{x_i})}{u_{x_i}}) \longrightarrow 1-(\beta+\sigma_s)$ in $\{u_{x_i} \neq 0\}$ and blow up other wise. So using Fatou's Lemma, we can deduct the following corollary if we consider a fixed $x$. 
\begin{corollary}\label{measofcrit}
Let  $u \in C^{1}(\Bar{\Omega})$ be a weak solution of (\ref{1.1}), then $|Z_u|=0$ and for $s=p,q$ and  every fixed $x \in \Omega$, 
\begin{equation*}
    \begin{aligned}
    \int_{\Omega} \frac{|Du|^{s-2-\beta}|u_{x_{i}x_{j}}|^2}{|x-y|^\gamma} dy \leq C
    \end{aligned}
\end{equation*}
where $\beta<1$ ,$\gamma<
N-2$ if $N\geq 3$ and $\gamma=0 $ if $N=2$. In particular 
\begin{equation*}
    \begin{aligned}
    \int_{\Omega} |Du|^{s-2-\beta}|u_{x_{i}x_{j}}|^2 \leq C
    \end{aligned}
\end{equation*}
\end{corollary}
We can conclude our analysis once we prove the following result.
\begin{corollary}
If $u \in C^1(\Bar{\Omega})$ is a weak solution of \eqref{1.1}  then $Z_u \cap \partial\Omega=\emptyset ,$ $|Du|^{s-2}Du \in W^{1,2}(\Omega, \mathbb{R}^N)$, therefore $|Du|^{s-1} \in W^{1,2}(\Omega)$ for $s=p$ or $q$.
\end{corollary}
\begin{proof}
  Only using test functions with compact support $\Omega\setminus Z_u$ in lemma \ref{lu} we get 
\begin{equation}\label{est}
    \begin{aligned}
     \int_{E\setminus Z_u}|Du|^{s-2-\beta} \| D^2u\|^2 dx \leq C
    \end{aligned}
\end{equation}
where $\beta < 1$ and E is any compact set contained in $\Omega$ for $C\geq0$. So by taking  $$\phi_n \equiv G_{\frac{1}{n}}(|Du|^{s-2}u_{x_i}),$$ where $G_{\frac{1}{n}}$ is defined in \eqref{defG}, we have 
 \begin{equation*}
     \begin{aligned}
       \frac{\partial}{\partial x_j} \phi_n  
      :=  G_{\frac{1}{n}}'(|Du|^{s-2}u_{x_i})\frac{\partial}{\partial  x_j}(|Du|^{s-2}u_{x_i}) & 
       \\
      = G_{\frac{1}{n}}'(|Du|^{s-2}u_{x_i}) \bigg(|Du|^{s-2}\Tilde{u}_{ij}&+(s-2)|Du|^{s-4}(Du,\Tilde{D}u_j)u_{x_i}\bigg) 
     \end{aligned}
 \end{equation*}  
Note that $\phi_n \in W^{1,2}(E)$. Utilizing \eqref{est}, along with the fact that $u \in C^{2}(\Omega\setminus Z_u)$ and Definition \ref{grad}, we can obtain a uniform bound $K\geq0$ such that $\|\phi_n\|_{W^{1,2}(E)} \leq K$ for each $n\in\mathbb{N}$. Since $W^{1,2}(E) \dhookrightarrow L^2(E)$ we have $\phi_n \rightarrow h$ strongly in $L^2(E)$ upto some subsequence where $h\in W^{1,2}(E)$. Then we can get that the convergence is almost everywhere in $E$. Now $\phi_n \longrightarrow |Du|^{s-2}u_{x_i}$  a.e. in $E$ hence $|Du|^{s-2}Du \in W_{loc}^{1,2}(\Omega, \mathbb{R}^N)$. By Hopf's lemma we have $Z_u \cap \partial \Omega = \emptyset$. Also we know that $u\in C^2(\Bar{\Omega}\setminus Z_u)$. So we get the regularity of $u$ near the boundary. Hence we can conclude $|Du|^{s-2}Du \in W^{1,2}(\Omega, \mathbb{R}^N)$. 
\end{proof}
We shall conclude this section by discussing the connectedness of $\Omega \setminus Z_u$.
\begin{theorem}
Let $u \in C^1(\Bar{\Omega})$ be a weak solution of \eqref{1.1}. Then $\Omega \setminus Z_u$ is connected i.e. $\Omega \setminus Z_u$ does not contain any connected component $\mathcal{C}$ such that $\Bar{\mathcal{C}} \subset \Omega$.
\end{theorem}
\begin{proof}
  If $\mathcal{C}$ be a connected component of $\Omega \setminus Z_u$ such that $\mathcal{C} \subset\subset \Omega$ then $Du(x)=0 \ \forall \,  x \in \partial \mathcal{C} .$ Exploiting the $H_a$ property of $a(x)$ and the fact that $|Du|^{s-2} Du \in W_0^{1,2}(\mathcal{C},\mathbb{R}^N)$ for $s=p$ or $q$, we could get the existence of $A_n, B_n \in C_0^\infty(\mathcal{C},\mathbb{R}^N)$ such that
  \begin{equation*}
      \begin{aligned}
      & A_n \longrightarrow |Du|^{p-2} Du
      \\
      & B_n \longrightarrow a(x)|Du|^{q-2} Du
      \end{aligned}
  \end{equation*}
  in the norm of $W_0^{1,2}(\mathcal{C}, \mathbb{R}^N).$ If $\supp A_n \subset \subset E_1 \subset \subset \mathcal{C}$ and $\supp B_n \subset \subset E_2 \subset \subset \mathcal{C}$ for $E_1,E_2$ being smooth subsets then using divergence theorem we have the following for an outside normal $\mathcal{n}$ and $\phi \in W^{1,2}$
  \begin{equation*}
      \begin{aligned}
          \int_\mathcal{C} div (A_n+B_n)\phi + (A_n+B_n,D \phi) dx
     =\int_{\partial E_1}\phi(A_n,\mathcal{n}) d\sigma+\int_{\partial E_2}\phi(B_n,\mathcal{n}) d\sigma =0
      \end{aligned}
  \end{equation*}
 Since $|Z_u|=0$ we get that $-div(|\nabla u|^{p-2}\nabla u+ a(x) |\nabla u|^{q-2}\nabla u)  = \frac{g(x)}{u^\delta}+h(x)f(u) $ a.e. in $\mathcal{C}$. Choose $\phi \equiv k \neq 0$ to get 
 \begin{equation*}
     \begin{aligned}
         \int_\mathcal{C} \frac{g(x)}{u^\delta}+h(x)f(u) \phi &= \int_\mathcal{C} -div(|\nabla u|^{p-2}\nabla u+ a(x) |\nabla u|^{q-2}\nabla u) \phi 
         \\
         &= \lim_{n \longrightarrow \infty } \int_\mathcal{C} -div (A_n+B_n) \phi =0
     \end{aligned}
 \end{equation*}
 This gives a contradiction to the fact that  $\frac{g(x)}{u^\delta}+h(x)f(u)$ is positive. Thanks to Hopf's lemma, a neighbourhood of boundary belongs to  $\mathcal{C}$ of $\Omega\setminus Z_u$. If there exists another compactly contained component $\mathcal{C}'$ then it would be a contradiction to what we have just proved. Hence $\Omega \setminus Z_u$ is connected. 
\end{proof}
\section{Symmetry And Monotonicity}\label{symandmon}
In this section, we prove the main theorem, which is the symmetry result. The existence of $C^{1,\alpha}(\Bar{\Omega})$ solutions to problem $\eqref{1.1}$ is guaranteed by Assumption \ref{h1}, as shown in \cite{GKS21}. Here, we provide a detailed discussion of the moving plane method. Without loss of generality, we demonstrate the symmetry and monotonicity results in the $x_1$ direction. For a ball, we can apply the same analysis in any given direction, which yields the radial and radially decreasing nature of the solutions. We begin with introducing some notations.
\\
\underline{\textit{Notations:}}
\\
We define $a:=\inf_{x \in \Omega} \ x_1$. For $\lambda \in \mathbb{R}$ with $ a < \lambda< 0$ we set $\Omega_\lambda:=\{x \in \Omega:x_1< \lambda \}$. The reflection of $x \in \Omega_\lambda$ with respect to $T_\lambda:=\{ x \in \mathbb{R}^N :  x_1=\lambda \}$ is denoted by $x_\lambda:=R_\lambda(x)=(2\lambda-x_1,x_2,...,x_n)$. We set $u_\lambda(x)=u(x_\lambda)$. Finally we define 
\begin{equation}
    \begin{aligned}
        \Lambda_0:=\{a<\lambda<0:u \leq u_t \ \text{in} \ \Omega_t \ \text{for \ all \ } t\in(a,\lambda] \}
    \end{aligned}
\end{equation}
\newpage
\underline{\textbf{Proof of Theorem \ref{main}}}
\begin{proof}
 Our first claim is that $\Lambda_0 \neq \emptyset$. Using theorem \ref{hpfe}, we can establish that $\frac{\partial u}{\partial x_1}>0$ in $\Omega_\lambda\cup R_\lambda(\Omega_\lambda)$ for small $\lambda-a$. Consequently, $u<u_\lambda$ in $\Omega_\lambda$ (near the boundary). This implies that $\Lambda_0$ is not an empty set when $\lambda-a$ is small. We can then define $\lambda_0=\sup \Lambda_0$.
\vspace{0.4cm}
\\
\underline{\textsc{Claim:}}
  $\lambda_0=0$ i.e. $u \leq u_\lambda$ in $\Omega_\lambda$ for every $\lambda \in (a,0]$. By contradiction, if $\lambda_0<0$ then we shall show $u \leq u_{\lambda_0+\tau}$ in $\Omega_{\lambda_0+\tau}$ for any $0<\tau<\Bar{\tau}$ for some $\Bar{\tau}>0$ (small).
\vspace{0.4cm}
\\ 
We get $u \leq u_{\lambda_0}$ in $\Omega_{\lambda_0}$ using continuity. It implies that $ h(x)f(u(x))\leq h(x_{\lambda_0})f(u(x_{\lambda_0}))$ by $H_h$ and $H_f$ of Assumption \ref{h2} . Using the mean value theorem and $H_g$, for some $C \geq 0$ we get
\begin{equation*}
    \begin{aligned}
    \frac{g(x)}{u(x)^\delta} - \frac{g(x_{\lambda_0})}{u(x_{\lambda_0})^\delta}  
\leq & g(x_{\lambda_0}) \big{(} \frac{1}{u(x)^\delta} - \frac{1}{u(x_{\lambda_0})^\delta}\big{)}  \leq C (u(x_{\lambda_0})-u(x)). 
    \end{aligned}
\end{equation*}
So we can use SCP (Theorem \ref{scp}) for $\Omega\setminus Z_u$ to conclude that $u < u_{\lambda_0}$ in $\Omega_{\lambda_0}\setminus Z_u$ unless $u \not\equiv u_{\lambda_0}$. Next, we will elaborate on the reasons why the last situation is not feasible.
\\
Let $\mathcal{C}$ be any connected component of $\Omega\setminus Z_u$. If $\partial\mathcal{C}\cap\partial\Omega\neq\emptyset$ then $u \equiv u_\lambda$ is not possible due to the zero Dirichlet boundary condition and positivity of $u$ in the interior of the domain. Else if $\partial\mathcal{C}\cap\partial\Omega=\emptyset$, then we get a contradiction to the fact that $\Omega\setminus Z_u$ is connected using local symmetry region. 
\\
Using uniform continuity, it can be shown that for any compact set $K\subset \Omega_{\lambda_0}\setminus Z_u$ and any small $\overline{\tau}>0$, $u<u_{\lambda_0 +\tau}$ in $K$ for $0<\tau<\overline{\tau}$. By applying theorem \ref{hpfe}, we can find a  $\delta>0$, such that $u <u_{\lambda_0 +\tau}$ in $I_\delta(\partial \Omega) \cap \Omega_{\lambda_0+\tau}$. Furthermore, as a result of Hopf's Lemma (Theorem \ref{hpfe}), we can also achieve the same outcome in a neighborhood of $\partial \Omega \cap T_{\lambda_0+\tau}$, following a similar idea that in \cite[Page no 20]{ES20}.
\\
Since $|Z_u|=0$ (Corollary \ref{measofcrit}), there exists a compact set $K$ large enough and $\tau$ small enough so that $|\Omega_{\lambda_0+\tau}\setminus K | < \tau$ for any fixed $\tau>0$. Setting $w_{\lambda_0+\tau}:=(u - u_{\lambda_0+\tau} )^+$ for any $0<\tau<\overline{\tau}$ we have $\supp w_{\lambda_0+\tau}\subset\subset \Omega_{\lambda_0+\tau}$ for $\Bar{\tau}$ small. Moreover, $w_{\lambda_0+\tau}=0$ in $K$ by construction. Take $\Bar{\tau}$ sufficiently small in order to guarantee the applicability of weak  comparison  principle for narrow domain (Theorem \ref{wcp}).
\\
Now using $H_g$ and $H_h$ (Assumption \ref{h2}), we have
\begin{equation*}
    \begin{aligned}
    \frac{g(x)}{u(x)^\delta} -\frac{g(x_\lambda)}{u(x_\lambda)^\delta}+h(x)f(u(x))-h(x_\lambda)f(u(x_\lambda)) \leq g(x_\lambda)\Bigg( \frac{1}{u^\delta}-\frac{1}{u_\lambda^\delta}&\Bigg)+h(x_\lambda) \bigg( f(u)-f(u_\lambda)\bigg)
    \\
    &
    \leq \|h_\lambda\|_\infty \bigg( f(u)-f(u_\lambda)\bigg).
    \end{aligned}
\end{equation*}
Here we use the facts that $\int_{[u \geq u_\lambda]} g(x_\lambda)\big( \frac{1}{u^\delta}-\frac{1}{u_\lambda^\delta}\big)(u-u_\lambda) \leq 0$ and $0 < h(x_\lambda)\leq \|h_\lambda\|_\infty$. 
By applying theorem \ref{wcp}, we can infer that $w_{\lambda_0+\tau}=0$  in $\Omega_{\lambda_0+\tau}$ for any $0<\tau<\overline{\tau}$, where $\overline{\tau}>0$ is small. This contradicts the definition of $\lambda_0$, so we can conclude that $\lambda_0$ must be equal to zero. Consequently, we can obtain the desired symmetry result in the $x_1$ direction by applying the moving plane procedure in the opposite direction.
Lemma \ref{hpfe} will imply that $\frac{\partial u}{\partial x_1}<0$, that gives the monotonicity in $x_1$ direction.
\end{proof}
\section{Strong Comparison Principle}\label{SCPsec}
In this section, we are looking for an SCP of the following equations. 
\begin{equation} \label{e1}
    \begin{aligned}
    -\Delta_p u_1 -\Delta_q u_1= \frac{c}{u_{1}^\delta}+h_1(x) \; \text{in} \; \Omega \hspace{0.49cm}\text{and}\hspace{0.49cm} u_1=0 \; \text{on} \; \partial\Omega
    \end{aligned}
\end{equation}
\begin{equation} \label{e2}
    \begin{aligned}
    -\Delta_p u_2 -\Delta_q u_2= \frac{c}{u_{2}^\delta}+h_2(x)  \; \text{in} \; \Omega \hspace{0.49cm}\text{and}\hspace{0.49cm} u_2=0 \; \text{on} \; \partial\Omega 
    \end{aligned}
\end{equation}
where $\Omega \equiv B_R(0)$ and $h_1, h_2 $ are continuous radial functions on $B_R$. We assume $0 \leq h_1 \leq h_2$  but $h_1 \not \equiv h_2$ in $\Omega$ and $c>0$. The existence of the positive solutions is guaranteed in \cite{GKS21}. Since $h_1 \leq h_2$, we have $u_1 \leq u_2$ using test function method. We are interested in proving $0< u_1 < u_2$ in $\Omega$ and $\frac{\partial u_1}{\partial \nu}> \frac{\partial u_2}{\partial \nu}$ on $\partial\Omega$. Here we use a technique similar to that in \cite{DIJ23}. For the sake of completeness, we briefly discuss the underlying idea of the proof. 
\\
\underline{\textbf{Proof of Theorem \ref{SCPfleqg}}}
  \begin{proof}
 We  define  $w=u_2-u_1$ and  $U_{r}:= B_{R}\setminus B_{r}$ for any $0<r<R$ where $B_R(0)\equiv B_R$. We rewrite the equations \eqref{e1} and \eqref{e2} as follows
    \begin{equation}
    \left\{
        \begin{aligned}
        -div(A(x)\nabla w)- \lambda B(x)w & \geq h_2-h_1 \geq 0 \hspace{3mm} \text{in}\ U_{r}
        \\
		w& \geq  0 \hspace{2.1cm} \text{on}\ \partial U_{r}
        \end{aligned}
        \right.
    \end{equation}
    for a matrix $A(x)=[a_{ij}(x)]$ and a scalar function $B(x).$ Here, 
    \begin{equation*}
        \begin{aligned}
        a_{ij}(x)=\int_{0}^{1} \sum_{s=p,q} |(1-t)\nabla u_1(x)+t\nabla u_2(x)|^{s-2}\Big\{&\delta_{ij}
        \\+(s-2)&\frac{((1-t)u_{1_{i}}+tu_{2_{i}})((1-t)u_{1_{j}}+tu_{2_{j}})}{|(1-t)\nabla u_1(x)+t\nabla u_2(x)|^{2}}\Big\}dt\\ \nonumber
        \end{aligned}
    \end{equation*}
    and
    \begin{equation*}
        \begin{aligned}
        B(x)=-\delta \int_{0}^{1}\frac{dt}{((1-t)u_1(x)+tu_2(x))^{\delta +1}}
        \end{aligned}
    \end{equation*}
	To write in this form, we use the fact that $\frac{du_i}{dr}<0$ for $i=1,2$ near the boundary. Under the given assumption on $u_1$ and $u_2$ we get that $A(x)$ is uniformly elliptic in $U_r$ for each $r>0$. Using the continuity of $h_1$ and $h_2$, we can confirm the existence of $r_0>0$ such that $h_1-h_2\not\equiv 0$ in $U_{r_{0}}$ and hence $w>0$ in $U_r$ for all $r<r_0$  using \cite[Theorem 2.5.2]{PS07}.
	\\
	Now we shall show the positivity of $w$ in $B_{r_0}$. Using the radial symmetry of the solutions the problem \eqref{e1} can be reduced to a system of ODEs:
			\begin{equation}
   \left\{ 
   \begin{aligned}
	u_{11}^{'}=\alpha(r,u_{12}),\; u_{11}(r_{1})=u_{11,0}\\
			u_{12}^{'}=- \frac{N-1}{r}u_{12}+B_1(r,u_{11}),\; u_{12}(r_{1})=u_{12,0}
    \end{aligned}
    \right.
\end{equation}
		where $r_{1} \in (0,R)$, $u_{11}(r)=u_1(r)$, $u_{12}(r)=|u_1^{'}(r)|^{p-2}u_1^{'}(r)+|u_1^{'}(r)|^{q-2}u_1^{'}(r)$, $B_1(r,y)= -(\frac{c}{y^{\delta}}+h_1(r)) $. The function
		$\alpha(r,y):(0,R)\times \mathbb{R}\rightarrow \mathbb{R}$ is the inverse of the function $a(r,y)=|y|^{p-2}y+|y|^{q-2}y.$
		Clearly $u_{11}(R)=u_{12}(0)=0$. Analogously we can write the equation \eqref{e2} as
		\begin{equation}
   \left\{ 
   \begin{aligned}
	u_{21}^{'}=\alpha(r,u_{22}),\; u_{21}(r_{1})=u_{21,0}\\
			u_{22}^{'}=- \frac{N-1}{r}u_{22}+B_2(r,u_{21}),\; u_{22}(r_{1})=u_{22,0}
    \end{aligned}
    \right.
\end{equation}
where $u_{21}(R)=u_{22}(0)=0$ and $B_2(r,y)= -(\frac{c}{y^{\delta}}+h_2(r)) $. Since $a(y)$ is strictly monotone and continuous we have $\frac{\partial \alpha}{\partial y} \geq 0$. 
\\
Suppose that there exists some $r'<r_0$ such that $w(r')=0$ i.e. $w$ attains its minimum at $r'$ and hence $w'(r')=0$. Taking $r_1=r'$ in the systems of ODEs, 
$u_{11,0}= u_{21,0}$ and $u_{12,0}=u_{22,0}$. For the function $b(x,u)=c u^{-\delta}$, we have $0\leq -\frac{\partial b}{\partial u}\in L^{\infty}_{loc}((-R,R)\times (0,\infty)),$ and hence by using \cite[Lemma 3.2]{CT00}, we obtain $u_{2}(r) \leq u_{1}(r)\; \forall \; r \in [r_{1},R)$ which contradicts the fact that $w>0$ in $U_{r_0}. $ Therefore, $0<u_1<u_2$ in $B_{R}.$ Finally we note that since $w>0$ in $B_R$ we can apply \cite[Theorem 2.7.1]{PS07} to conclude that
		$\frac{\partial u_2}{\partial \nu}<\frac{\partial u_1}{\partial \nu}<0.$
  \end{proof}
We note that the hypothesis of Theorem \ref{SCPfleqg} could be modified as below and still the SCP holds.    
 \begin{theorem}\label{SCPfleqg1}
 Let $1<p \leq q \leq 2$ and  $h_1$, $h_2$ are continuous radial functions on $\Omega\equiv B_R(x_0)$ for $x_0\in \mathbb{R}^N$such that $ h_1 \leq h_2$  but $h_1 \not \equiv h_2$ in $\Omega$. Assume that $u_1,u_2 \in C^{1,\alpha}(\Bar{\Omega})$ such that $u_1\leq u_2$ and they are radially decreasing positive solutions of 
\begin{equation}\label{scpeqn}
    \begin{aligned}
            -\Delta_p u_1 -\Delta_q u_1= \frac{c}{u_{1}^\delta}+h_1(x) \; \text{in} \; \Omega \hspace{0.49cm}\text{and}\hspace{0.49cm} u_1=0 \; \text{on} \; \partial\Omega
            \\
                -\Delta_p u_2 -\Delta_q u_2= \frac{c}{u_{2}^\delta}+h_2(x))  \; \text{in} \; \Omega \hspace{0.49cm}\text{and}\hspace{0.49cm} u_2=0 \; \text{on} \; \partial\Omega
    \end{aligned}
\end{equation}
Then $ u_1 < u_2$  and $\frac{\partial u_2}{\partial \nu}<\frac{\partial u_1}{\partial \nu}<0$ in $\Omega$.
 \end{theorem}
We conclude our manuscript by commenting on the range of $p$ and $q$ for which the SCP holds. 
  \begin{remark}
  We remark here that the assumption $1<p\leq q\leq 2$ is necessary in proving the SCP in the general setting. We draw the attention of the readers to the counterexample provided in \cite[Page no 4]{DIJ23} where we show that SCP is violated when $p=q>2$ and the domain is a open ball in $\mathbb{R}^N.$ 
  \end{remark}
  \bibliographystyle{plain}
 \bibliography{reference}
  
	\end{document}